\documentclass{amsart}

\usepackage{amsfonts, amsthm, amsmath, amscd, amssymb}
\usepackage[all]{xy}
\usepackage[mathscr]{eucal}
\usepackage{epic, eepic}
\usepackage{enumerate}

\newtheorem{theorem}{Theorem}[section]
\newtheorem{lemma}[theorem]{Lemma}
\newtheorem{proposition}[theorem]{Proposition}
\newtheorem{corollary}[theorem]{Corollary}

\theoremstyle{remark}
\newtheorem{remark}[theorem]{Remark}

\numberwithin{equation}{section}

\newcommand{\mbar}{\overline{\M}}
\newcommand{\h}{\mathcal{H}}
\newcommand{\M}{\mathcal{M}}
\newcommand{\I}{\mathcal{I}}
\newcommand{\so}{\mathcal{O}}
\newcommand{\X}{\mathcal{X}}
\newcommand{\Y}{\mathcal{Y}}

\newcommand{\cc}{\mathbb{C}}
\newcommand{\pp}{\mathbb{P}}

\newcommand{\tch}{\widetilde{Ch}}
\newcommand{\ttd}{\widetilde{Td}}

\newcommand{\ld}{\lambda}
\newcommand{\ep}{\epsilon}

\newcommand{\ol}{\overline}
\newcommand{\lp}{\left(}
\newcommand{\rp}{\right)}

\newcommand{\hot}{ \text{higher degree terms}}
\DeclareMathOperator{\spec}{Spec}

\DeclareMathOperator{\id}{Id}
\DeclareMathOperator{\supp}{Supp}

\begin{document}

\title[Euler characteristics of cotangent line bundles]
{Euler characteristics of \\
universal cotangent line bundles on $\mbar_{1,n}$}

\author{Y.P.~Lee}
\email{yplee@math.utah.edu}

\author{F.~Qu}
\email{qu@math.utah.edu}

\address{Department of Mathematics, University of Utah,
Salt Lake City, Utah 84112-0090, U.S.A.}
\thanks{This project was partially supported by the NSF}


\commby{Lev Borisov}

\begin{abstract}
We give an effective algorithm to compute the Euler characteristics
$\chi(\mbar_{1,n}, \otimes_{i=1}^n  L_i^{ d_i})$.
This work is a sequel to \cite{L0}.

In addition, we give a simple proof of Pandharipande's vanishing theorem 
$H^j (\mbar_{0,n}, \otimes_{i=1}^n L_i^{ d_i})=0$ for $j \ge 1, d_i \ge0$.
\end{abstract}

\maketitle

\setcounter{section}{-1}

\section{Introduction}
Let $\mbar_{1,n}$ be the moduli stack of $n$-pointed genus $1$ stable curves,
$\so$ its  structure sheaf, 
$\h$ the Hodge bundle, and  $L_i$
the universal cotangent line bundles at the $i$-th marked point, $ 1 \le i \le n$. 
The main result of this paper is the following theorem.

\begin{theorem} \label{t:main}
There is an effective algorithm of computing the Euler characteristics
\[ 
 \chi_{d, d_1, \ldots, d_n} := 
  \chi(\mbar_{1,n}, \h^{ -d} \otimes \bigotimes_{i=1}^n L_i^{ d_i} ), \qquad 
 d,\ d_i \ge 0. 
\]
\end{theorem}
\noindent
The details of this algorithm are presented in Section~2.

\vspace{10pt}

This work is a sequel to \cite{L0}, where we calculated the 
Euler characteristics
\[
  \chi (\mbar_{0,n}, \otimes_i L_i^{ d_i})
\]
at genus zero.
These are our preliminary attempts in search of a $K$-theoretic version of
Witten--Kontsevich's theory of two-dimensional topological gravity.
In the Witten-Kontsevich theory, the correlators are the intersection numbers
of tautological classes on the moduli spaces of stable curves
\begin{equation} \label{e:wk}
 \int_{\mbar_{g,n}} \psi_1^{d_1} \ldots \psi_n^{d_n}.
\end{equation}
The natural $K$-theoretic version of intersection numbers 
(i.e.~pushforward to a point in cohomology theory)
are the Euler characteristics (i.e.~pushforward to a point in $K$-theory).
As Witten--Kontsevich theory states that a generating function
of \eqref{e:wk} is the $\tau$-function of the $KdV$ hierarchy,
it is reasonable to surmise that a similar generating function
in $K$-theory could be a $\tau$-function of a version of a discrete
$KdV$ hierarchy.
Note that the phenomenon of replacing differential equations in quantum
cohomology by finite difference equations in quantum $K$-theory were
observed in earlier examples \cite{GL} and only very recently 
demonstrated to hold in general \cite{GT}.

Furthermore, since Witten--Kontsevich theory is the Gromov--Witten theory
for the target space being a point, it is natural to consider these
calculations as basic ingredients for the \emph{quantum $K$-theory} 
developed in \cite{aG} and \cite{L}.
In the calculation of quantum $K$-invariants at genus one
via localization, the Hodge bundle will appear naturally.
It is therefore reasonable to consider slightly more general 
correlators, which have the additional benefits of
facilitating the induction process of our algorithm.

\vspace{10pt}

Our strategy of proving Theorem~\ref{t:main} is to apply
the orbifold Riemann-Roch theorem to
\[  
  \chi':= \chi \lp \mbar_{1,n}, \h^{ -d} \bigotimes_{i=1}^n (L_i^{ d_i}-\so) 
  \rp   , \qquad d, d_i \ge 0 .\]

We were able to determine $\chi'$ by carefully examining and performing computations on the twisted sectors of
$\mbar_{1,n}$.
In doing so, we find the use of \emph{generating functions} essential.
These functions can be found in Section~2, starting with Equation~\eqref{e:2.1}.
It is then not difficult to see that
one can determine $\chi$ on $\mbar_{1,n}$ by
$\chi'$ and  $\chi$ on $\mbar_{1,n-1}$.
Hence, we can reduce all calculations to $\mbar_{1,1}$,
whose generating function is calculated explicitly 
in Lemma~\ref{l:mf} and Proposition~\ref{p:gmf}.
Note that when $n=1$ the generating function is a rational function,
as in the case of genus zero case.
We expect the generating function of $\chi_{d, d_1, \ldots, d_n}$ to be
rational as well, but are not able to find the correct closed form.
We did, however, perform a consistency check: Our algorithm produces 
$\chi_{d, d_1, \ldots, d_n}$ as \emph{integers}, even though the 
intermediate steps require rational numbers, which arise as a consequence
of the orbifold Riemann--Roch formula and the stack structure of $\mbar_{1,n}$.

\vspace{10pt}

Indeed, the $n=1$ case is closely related to the theory of \emph{modular forms}.
Theorem~\ref{t:main} can be considered as a generalization of the 
following well-known fact.

\begin{proposition}[See Lemma~\ref{l:mf} and Proposition~\ref{p:gmf}] 
\label{p:mf}
\[
 \chi \left( \mbar_{1,1}, \frac{1}{1-q L_1} \right)
  =\frac{1}{(1-q^4)(1-q^6)}.
\]
\end{proposition}

Since $\mbar_{1,1}$ is the moduli stack of elliptic curves, and 
sections of $L_1^{ k}$ are the modular forms of weight $2k$,
Proposition~\ref{p:mf} can be considered as a rephrase of the classical result 
that the space of the modular forms is generated by a weight four and a 
weight six modular forms.

\vspace{10pt}

Another result included in this paper, in Appendix~A, is a new proof
of Pandharipande's vanishing theorem \cite{rP} at genus zero.

\begin{theorem}[\cite{rP}] \label{t:rP}
\[
 H^j (\mbar_{0,n}, \otimes_{i=1}^n L_i^{d_i}) =0
\]
for $j \ge 1$ and $d_i \ge 0$. 
\end{theorem}

Our proof is comparably simpler and shorter, and does not use
M.~Kapranov's results on $\mbar_{0,n}$ \cite{mK}.
Only basic definitions and elementary manipulation of spectral sequences
are used.

\vspace{10pt}

This paper is organized as follows. 
In Section~1 we recall the necessary background.
We then formulate a more precise version of the reduction algorithm
in Section~2, and prove Theorem~\ref{t:main} there.
In Appendix~A the (new) proof of Theorem~\ref{t:rP} is given.


\section{Preliminaries}
We work over the ground field $\cc$.

\subsection{Twisted sectors of  $\mbar_{1,n}$}
 We summarize the results we need concerning the inertia stack of $\mbar_{1,n}$
in \cite[section 3.b]{P1}.

For a DM stack $\X$, recall a $\spec\cc$ point of its inertia stack $I\X$
is given by a pair $(x,g)$ with $x \in \X(\spec\cc)$ and $g \in Aut_{\X}(x)$. 
$\X$ is naturally viewed as a component of $I\X$ consists of point $(x,g)$ with $g$ trivial, and we denote this component by $(\X, \id)$. A twisted sector is a connected component of $I\X$ disjoint from $(\X, \id)$. 

\begin{theorem}[\cite{P1} Theorem 3.22, 3.24]\label{inertia}
The twisted sectors of $I\mbar_{1,n}$ come from either of the following
two sources:
\begin{enumerate}
   \item 
      the closure of a twisted sector of $I\M_{1,n}$ in $I\mbar_{1,n}$, 
   \item 
      a twisted sector of 
     $I(\mbar_{0, K\cup \bullet} \times \mbar_{1, K^c \cup \bullet})$
      via $I\Delta_K$. 
\end{enumerate}
Here 
$
 \Delta_{K}:\mbar_{0, K\cup \bullet} \times \mbar_{1, K^c \cup \bullet} 
  \to \mbar_ {1,n}
$
is the closed immersion gluing the marked points $\bullet$,
$K$ is a subset of $[n]$
with $|K| \ge 2$, and $K^c$ its complement.
$
   I\Delta_K:
   I(\mbar_{0, K\cup \bullet} \times \mbar_{1, K^c \cup \bullet}) 
   \to I\mbar_{1,n}
$
is the induced closed immersion between the corresponding inertia stacks.
\end{theorem}

As $ I(\mbar_{0, K\cup \bullet} \times \mbar_{1, K^c \cup \bullet}) \simeq 
   \mbar_{0, K\cup \bullet} \times I\mbar_{1, K^c \cup \bullet} $, type (2) twisted sectors are built up from
type (1) twisted sectors.

The analysis of type (1) twisted sectors in \cite{P1}
starts with the following description of $\mbar_{1,1}$.
\begin{proposition}
    $\mbar_{1,1}$ is equivalent to the weighted projective space $\pp(4,6)$.
 \end{proposition}

We briefly recall the equivalence appeared in \cite[Theorem 3.8]{P1},   as we will need it to do explicit
calculations on $\mbar_{1,1}$, and it also serves to motivate the notations used
in \cite{P1}(Notation 3.9,~3.12;  Definition 3.13,~3.16) that we follow.

Let $U=\mathbb{A}^2-(0,0)$ with $\cc^*$ action: 
$\ld \cdot (a,b) = (\ld^4a,\ld^6b)$,
where  $\ld \in \cc^*$, $(a,b) \in U$.
The equivalence from $\pp(4,6) :=[U/\cc^*]$ to $\mbar_{1,1}$ is induced from 
a $\cc^*$-equivariant family of 1 pointed genus one stable curves $C\to U$.
where 
$$C=\{ (a,b) \times [x: y :z] \in U \times \pp^2 |~ y^2z=x^3+axz^2+bz^3 \},$$
with the section 
$$s: U\to U\times[0,1,0] \subset C,$$
the $\cc^*$ action is given by
$$ \ld \cdot ((a,b) \times [x:y:z]) =(\ld^4a,\ld^6b) \times[\ld^2x: \ld^3y:z].$$

Denote by
\begin{itemize}
   \item 
      $A_k$:   the  component of $I\M_{1,k}$
     consisting of pairs $(x,g)$ with $g$ of order $2$, here $1 \le k \le 4 .$
   \item      
      $C_4$: the 1-pointed curve $\{ [x:y:z] \in \pp^2 |\ y^2z=x^3+xz^2 \}$ with
      $[0:1:0]$ marked. Its automorphism group is generated by 
      $i=\sqrt{-1}$.
   \item      
      $C_6$:  1-pointed curve $\{ [x:y:z] \in \pp^2 |\ y^2z=x^3 +z^3\}$ with 
      $[0:1:0]$ marked.
      Its automorphism group is generated by $\ep=\exp{(2\pi i/6)}$.
      
   \item 
      $C_4'$: $C_4$ with a $2$nd marked point $[0:0:1]$.
   \item 
      $C_6'$:  $C_6$ with a 2nd marked point $[0:1:1]$.
   \item 
      $C_6''$:  $C_6'$ with a 3rd marked point $[0:-1:1]$.
\end{itemize}

\begin{theorem}[\cite{P1} Corollary 3.14]\label{sm inertia} 
\hfill
\begin{enumerate}
   \item
      $I\M_{1,5}=(\M_{1,5}, \id), n\ge5 $.
   \item 
       For $k\le 4$, the twisted sectors of $I\M_{1,k}$ are of dimension 1 or 0.
        $ A_k$ is the only 1 dimensional twisted sector.    
       Zero dimensional twisted sectors are of the form $B\mu_r$, and
       they are determined  by
    \begin{itemize}
       \item   $(C_4, i), (C_4, -i),(C_6, \ep),
       (C_6, \ep^2),(C_6, \ep^4),(C_6, \ep^5)$ for $\M_{1,1}$.
       \item   $(C'_4, i),(C'_4, -i),(C'_6, \ep^2),(C'_6, \ep^4)$ for $\M_{1,2}$.
       \item  $(C''_6, \ep^2),(C''_6, \ep^4)$ for $\M_{1,3}$.
    \end{itemize}
\end{enumerate}
\end{theorem}

\begin{remark}
       Given $x \in \X(\spec\cc)$ and an order $r$ element 
        $g \in Aut_\X(x)$, the pair $(x,g)$ determines a representable morphism from 
        $B\mu_r$ to $\X$. (see~\cite{AGV} 3.2)          
   As $B\mu_r$ is proper, it is
    closed in $I\mbar_{1,k}$.
\end{remark}

\begin{theorem}
[\cite{P1} Corollary 3.11, lemma 3.17]
\label{basic}    
\hfill 

Let $\ol{A_k}$ be the closure of $A_k$ in $I\mbar_{1,k}$, then
\begin{enumerate}
   \item 
      \begin{itemize}
          \item
             $\ol{A_1}$ is isomorphic to $\mbar_{1,1}$.
          \item 
             $\ol{A_2}\subset I\mbar_{1,2}$ is isomorphic to $\pp(2,4) $.
           \item 
              $\ol{A_3}\subset I\mbar_{1,3}$ is isomorphic to $\pp(2,2)$.
           \item 
              $\ol{A_4} \subset I\mbar_{1,4}$ is isomorphic to $\pp(2,2)$.
       \end{itemize}
     \item
       When viewed as a closed substack of $\mbar_{1,k}$ , 
       $\ol{A_k}$ does not intersect with the boundary divisors $\Delta_K$ for 
       any $K \subset [k]$, here $2\le k \le 4$.
\end{enumerate} 
\end{theorem}

\subsection{Riemann-Roch formula for Stacks}
We recall the Riemann-Roch formula in a version needed for this paper, 
adopted from Appendix A of \cite{T}.

\begin{theorem}[\cite{Ka},\cite{To} Corollary 4.13]\label{rr}
Let $\X$ be a smooth, proper Deligne-Mumford stack 
with quasi-projective coarse moduli space,
$E$ a vector bundle on $\X$. 
Assume $\X$ has the resolution property, 
i.e. every coherent sheaf is a quotient of a vector bundle, 
then we have the following formula for the Euler characteristics of $E$ :
$$ \chi(\X, E)=\int_{I\X} \tch(E)\ttd(\X).$$
\end{theorem}

Here
\begin{itemize}

   \item 
      $I\X$ is the inertial stack of $\X$, with projection $p_{\X}: I\X\to\X$.
   \item 
      $\tch(E) \in H^*(I\X)$ is the Chern character of the bundle $\rho(p_{\X}^{*}E)$.
   \item 
      $\rho(F):=\sum_{\zeta} \zeta F^{(\zeta)}\in K^0(I\X)$, if $F=\oplus F^{(\zeta)}$ with $F^{(\zeta)}$ being the eigenbundle of $F$ with eigenvalue $\zeta$.
      
    \item 
       $\ttd(\X)=\frac{Td(I\X)}{Ch(\rho\circ \lambda_{-1}(N_{I\X/\X}^\vee))}$,
       where $Td$ and $Ch$ are the usual Todd class and Chern character.
       $N_{I\X/\X}$ is the normal bundle for $p$, 
       and $N^{\vee}$ is the dual of $N$.
    \item 
       $\lambda_{-1}(V):=\sum_{a\geq 0} (-1)^a\Lambda^aV$ is 
       the $\lambda_{-1}$ operation in K-theory. If $V=\oplus_i V_i$ is
       direct sum of line bundles $V_i$,  then 
       $\lambda_{-1}(V)=\prod_i (1 - V_i)$.
 
\end{itemize}

\begin{remark}
      $ \mbar_{1,n}$ satisfies the resolution property. 
  See, e.g.~\cite{Kr} Proposition 5.1.
 \end{remark}

\begin{remark} \label{pullback}
For a vector bundle $F$ over $\Y$ and a map $f: \X  \to \Y$ inducing
$If: I\X \to I\Y$,
it is easy to see $\rho( p_{\X}^*f^*(F)) = If^* (\rho(p_{\Y}^*F))$ in $K^0(I\X)$.
\end{remark}
\subsection{String Equation}

\begin{proposition}[\cite{L} Section 4.4] \label{se} 
Let $\pi :\mbar_{g,n} \to \mbar_{g,n-1}$ be the forgetful map forgetting the 
$n$-th marked point,  then in terms of generating functions with
variables  $q, q_i,1\le i <n$, we have, when $g=0$,

$$
\pi_*(\prod_{i<n} \frac{1}{1-q_i L_i} )= 
\prod_{i<n}\frac{1}{1-q_i L_i} \cdot
(1+ \sum _{i<n} \frac{q_i}{1-q_i}), 
$$
and when $g>0$,

$$
\pi_*(\frac{1}{1-q\h^{-1}}\prod_{i<n} \frac{1}{1-q_i L_i} )= 
\frac{1}{1-q\h^{-1}}\prod_{i<n}\frac{1}{1-q_i L_i} \cdot
(1 -\h^{-1}+ \sum _{i<n} \frac{q_i}{1-q_i}).
$$ 

Here 
$\pi_*$ is the \emph{$K$-theoretic} pushforward.
\end{proposition}

\section{Euler characteristics of universal cotangent line bundles }

In this section, we give an algorithm to compute
\[  \chi \lp \mbar_{1,n}, \h^{ -d} \bigotimes_{i=1}^{n} 
 L_i^{ d_i} \rp , d, d_i \ge 0.\]
It is more efficient to encode these numbers into a generating function
\begin{equation} \label{e:2.1}
 \begin{split}
  X_n := &\chi \lp \mbar_{1,n}, \frac{1}{1-q\h^{-1}} 
    \prod_{i=1}^{n}\frac{1}{1-q_i L_i} \rp \\
     =   &\sum_{d,d_i \ge 0} q^d \prod_{i=1}^{n}q_i^{d_i} 
    \chi \lp \mbar_{1,n}, \h^{ -d} 
     \bigotimes_{i=1}^{n} L_i^{ d_i} \rp.
 \end{split}
\end{equation}
   
We will first show  that the calculation of $X_n$ can be reduced to that of 
$X_{n-1}$ if another generating function $\Phi_n$ in \eqref{e:2.2}
can be calculated.
We then explicitly determine $\Phi_n$ and $X_1$.

\subsection{Reduction from $\mbar_{1,n}$ to $\mbar_{1,n-1}$} \label{s:2.1}
Let $\Phi_n$ be the generating function
\begin{equation} \label{e:2.2} 
   \Phi_n : = \chi
   \lp \mbar_{1,n}, \frac{1}{1-q\h^{-1}} \prod_{i=1}^n
   \lp \frac{1}{1-q_i L_i} - \frac{1}{1-q_i\so} \rp
   \rp\\
\end{equation}

\begin{lemma} When $n >1$, $X_n$
is determined by $\Phi_n$ and $X_{n-1}$. 
More precisely, we have
\begin{align*}
 X_n(q, q_1,\cdots, q_n)= &\Phi_n(q, q_1,\cdots, q_n)+
\sum_{I \subset [n], I\ne \emptyset} (-1)^{|I|+1} \prod_{i \in I} \frac{1}{1-q_i}\cdot\\
& \Big (
X_{n-1}(q, {\{q_j , j \notin I\}}, \underbrace{0, \cdots, 0}_{|I|-1})
(1- \frac{1}{q} + \sum_{j \notin I} \frac{q_j}{1-q_j}) \\
&+ \frac{1}{q}X_{n-1}(0, {\{q_j , j \notin I\}}, \underbrace{0, \cdots, 0}_{|I|-1})
\Big ).
\end{align*}
For the last line,  note that $X_{n-1}(q, q_1, \cdots, q_{n-1})$ is a symmetric function of the variables  $q_1, q_2, \cdots, q_{n-1}$, and it is evaluated at
${\{q_j , j \notin I\}}$ and $|I|-1$ zeros.

\end{lemma}

\begin{proof}
This follows directly from the definition and the string equation.
Expand the product $\prod_{i=1}^n 
(\frac{1}{1-q_i L_i} - \frac{1}{1-q_i\so} )$ in $\Phi_n$ we have
$$
\Phi_n
 =
X_n+ 
\sum_{I \subset [n], I \ne \emptyset} (-1)^{|I|} 
\prod_{i \in I} \frac{1}{1- q_i}
 \cdot 
   \chi 
   \lp
      \mbar_{1,n}, \frac{1}{1- q\h^{-1}} \prod_{j \notin I} \frac{1}{1- q_j L_j}
      \rp.
$$

By the string equation
\begin{align*}
&\chi 
\lp
\mbar_{1,n}, \frac{1}{1- q\h^{-1}} \prod_{j \notin I} \frac{1}{1- q_j L_j}
\rp
 \\
& =
\chi 
\lp
\mbar_{1,n-1},  \frac{1}{1- q\h^{-1}} \prod_{j \notin I} \frac{1}{1- q_j L_j}
\cdot (1- \h^{-1} +\sum_{j \notin I} \frac{q_j}{1-q_j})
\rp, \\
\end{align*}
which is  
\[
\begin{split}
 &X_{n-1}(q, {\{q_j , j \notin I\}},\underbrace{0, \cdots, 0}_{|I|-1})
(1- \frac{1}{q} + \sum_{j \notin I} \frac{q_j}{1-q_j}) \\
+ &\frac{1}{q}X_{n-1}(0, {\{q_j , j \notin I\}},\underbrace{0, \cdots, 0}_{|I|-1}).
\end{split}
\]
From here it is easy to see the lemma holds.
\end{proof}

\subsection{Calculation of $\Phi_n$}

By the Riemann-Roch formula (Theorem~\ref{rr}), we have
$$\Phi_n=\int _{I\mbar_{1,n}} \tch (\frac{1}{1-q\h^{-1}} \prod_{i=1}^n
    ( \frac{1}{1-q_i L_i} - \frac{1}{1-q_i\so}))
    \ttd(\mbar_{1,n}).$$

As the inertia stack ${I\mbar_{1,n}}$ is the disjoint union of the distinguished component $(\mbar_{1,n}, \id)$ and its twisted sectors,
the integral  is the sum of the contributions from these components.

\begin{proposition}
The contribution to $\Phi_n$ from  $(\mbar_{1,n} ,\id)$ is 
 $$\frac{(n-1)!}{24(1-q)}\prod_{i=1}^n \frac{q_i}{(1-q_i)^2}.$$
\end{proposition}
\begin{proof}
On  $(\mbar_{1,n} ,\id)$, $\tch$ and $\ttd$ are the same as $Ch$ and $Td$ respectively,

\[
\begin{split} &Ch(\frac{1}{1-q\h^{-1}}\prod_{i=1}^{n}
(\frac{1}{1-q_i L_i} - \frac{1}{1-q_i\so}) \\
 = &\frac{1}{1-q}\prod_{i=1}^{n} \frac{q_i}{(1-q_i)^2}
 \prod_{i=1}^{n}c_1(L_i)+\hot.
\end{split}
\]

Applying the dilaton equation
$$\int_{\mbar_{1,n}} c_1(L_1)\cdots c_1(L_n)=(n-1) \int_{\mbar_{1,n-1}} c_1(L_1)\cdots c_1(L_{n-1}),$$
and 
$$\int_{\mbar_{1,1}} c_1(L_1) =\frac{1}{24},
$$
we find that 
\[
 \begin{split}
  &\int_{\mbar_{1,n}} 
  Ch\lp \frac{1}{1-q\h^{-1}} \prod_{i=1}^n 
 (\frac{1}{1-q_i L_i} - \frac{1}{1-q_i\so} ) \rp Td(\mbar_{1,n}) \\
  = &\frac{(n-1)!}{24(1-q)}\prod_{i=1}^n \frac{q_i}{(1-q_i)^2}.
 \end{split}
\]

\end{proof}

\begin{proposition}
The contribution to $\Phi_n$ from a twisted sector of type (2) in 
Theorem~\ref{inertia}, i.e.~from $I\Delta_K$, is zero.
\end{proposition}

\begin{proof}
Recall such a twisted sector is the product of 
$\mbar_{0, K \cup \bullet}$
and a twisted sector $\I$ of \
$\mbar_{1, K^c \cup \bullet}$.
The natural map 
$  \mbar_{0, K \cup \bullet} \times \I \to \mbar_{1,n}$
factors through 
\[ \Delta_{K}:\mbar_{0, K\cup \bullet} \times \mbar_{1, K^c \cup \bullet} 
  \to \mbar_ {1,n} .\]
We quote some known results :  
\begin{itemize}

   \item 
      The dual of the normal bundle for $\Delta_K$ is 
      $pr_1^*(L_{\bullet}) \otimes pr_2^*(L_{\bullet})$.
      Here $pr_i$ is the projection of
      $\mbar_{0, K\cup \bullet} \times \mbar_{1, K^c \cup \bullet}$
      onto its $i$-th factor.
   \item 
      $\Delta_K^*(L_i)$ is $pr_1^*(L_i)$ for $ i \in K$,  and is $pr_2^*(L_i)$
      for $i \notin K$.
   \item 
      $\Delta_K^*(\h)=pr_2^*( \h)$.  
     
\end{itemize}
Using these results, it is then straightforward to see that pushing forward the integrand

$$ \tch (\frac{1}{1-q\h^{-1}} \prod_{i=1}^n
( \frac{1}{1-q_i L_i} - \frac{1}{1-q_i\so}))
    \ttd(\mbar_{1,n})$$ 
to $\mbar_{0,K\cup \bullet}$
gives us  a class 
which has  a factor  $\prod_{i \in K} c_1(L_i)$
coming from 
$\tch(\prod_{i \in K} (\frac{1}{1-q_i L_i} -\frac{1}{1 - q_i\so}))$.
As the degree of $\prod_{i \in K} c_1(L_i)$ exceeds the dimension of 
$\mbar_{0, K \cup \bullet}$, the contribution is zero.
\end{proof}


\begin{proposition} For $2 \le k\le 4$,
the contribution from $\ol{A_k}$ is 
$$(-1)^k \frac{1}{24(1+q)}
    \prod_{i=1}^k\frac{q_i}{1-q_i^2}
    \cdot (11+\frac{2q}{1+q} -\sum_{i=1}^n \frac{2q_i}{1+q_i})
    \cdot d_k ,$$
where $d_k$ is $6,6,3$ for $k=4,3,2$, respectively. 
The number $d_k$ is the degree of 
a maps $\ol{A_k} \to \mbar_{1,1}$ forgetting all but one marked point.
\end{proposition}

\begin{proof}
On $\ol{A_k}$, we have
\begin{align*}
& \quad \tch( 
      \frac{1}{1-q\h^{-1} }  
      \prod_{i=1}^k ( \frac{1}{1-q_i L_i}  -  \frac{1}{1-q_i\so})  )
     \\
     &=\frac{1}{1+q e^{c_1(\h^{-1})} } \prod_{i=1}^k(\frac{1}{1+q_i e^{c_1(L_i)}} -
        \frac{1}{1-q_i})\\
      &=(-2)^k\frac{1}{1+q}\prod_{i=1}^k\frac{q_i}{1-q_i^2}\Big(1+\frac{q}{1+q}c_1(\h)
        + \sum_{i=1}^k\frac{1-q_i}{2(1+q_i)}c_1(L_i)\Big),\\   
\end{align*}
and
\[
    Ch\big(\rho\circ(\Lambda_{-1}(N_{\ol{A}_k/\mbar_{1,k}}^{\vee})\big)\big)=
          2^{k-1}\big(1+\frac{1}{2}c_1(N_{\ol{A}_k/\mbar_{1,k}}^{\vee})\big).
\]

For the above equations, note that over $\ol{A_k}$ the eigenvalues involved in $\tch$ must be $-1$ as the nontrivial  automorphism is of order 2,  also there is no higher degree terms as $\ol{A_k}$ is 1 dimensional.

Thus 
 \begin{align*}
   &\quad \int_{\ol{A_k}} \tch(\frac{1}{1-q\h^{-1}} \prod_{i=1}^k 
    (\frac{1}{1-q_i L_i} - \frac{1}{1-q_i\so}) )\ttd(\mbar_{1,k}))\\
   &=(-1)^k\frac{1}{1+q}\prod_{i=1}^k\frac{q_i}{1-q_i^2} \cdot \\
       &\ \ \quad \cdot \Big( 
         \frac{2q}{1+q} \int_{\ol{A_k}} c_1(H)
         + \sum_{i=1}^k  \frac{1- q_i}{1+q_i} \int_{\ol{A_k}} c_1(L_i) 
         + \int_{\ol{A_k}}c_1(T\mbar_{1,k})
         \Big ).
\end{align*}

It is easy to see$$\int_{\ol{A_k}} c_1(H)=d_k\int_{\mbar_{1,1}}c_1(H)=\frac{d_k}{24},$$ by considering a map
$\ol{A_k} \subset \mbar_{1,k }\to \mbar_{1,1}$ forgetting all but one marked point,
$$ \int_{\ol{A_k}} c_1(L_{i}),1\le i\le k, 
~\textrm{and}~
\int_{\ol{A_k}}  c_1(T\mbar_{1,k}).$$ are determined by 
Corollary \ref{cor}.

\end{proof}

\begin{lemma} Let $\pi: \mbar_{1,n+1} \to \mbar_{1,n}$ be the forgetful
map forgetting the $(n+1)$-th marked point, then
\begin{equation*}
c_1(L_j)=\pi^* c_1(L_j) + \Delta_{ \{ j , n+1\} }, 1 \le j \le n;
\end{equation*}
\begin{equation*}
c_1(T\mbar_{1,n +1})= 
\pi^* c_1(T\mbar_{1,n}) - c_1(L_{n+1}) +\sum_{1\le j \le n} \Delta_{\{j, n+1\}}.
\end{equation*}
\end{lemma}
\begin{proof}
Recall for 
 $\pi: \mbar_{1,n +1} \to \mbar_{1,n}$ 
we have
$$
L_j=\pi^* L_j \big( \Delta_{ \{ j, n+1 \} } \big), 1\le j\le n;\qquad
L_{n+1}=\omega_{\pi} \big( \sum_{1\le j \le n} \Delta_{\{j,n+1\}} \big),
$$

where
$
\omega_{\pi} = \omega_{\mbar_{1,n+1}} \otimes \pi^* \omega_{ \mbar_{1,n} }^{\ ~-1}
$ is the relative dualizing sheaf for $\pi$.
Taking the first Chern class of these equations proves the lemma.
\end{proof}

\begin{corollary} \label{cor}
\begin{equation*}
\int_{\ol{ A_k}} c_1(L_j)=
\frac{d_k}{24}, 1 \le j \le k. \quad
\int_{\ol{A_k}}  c_1(T\mbar_{1,k})=\frac{(11-k)d_k}{24}.
\end{equation*}
\end{corollary}
\begin{proof}
This can be easily proved by applying  the above lemma and Theorem \ref{basic} (2) to a forgetful map $\ol{A_k} \to \mbar_{1,1}$.

\end{proof}


\begin{proposition}\label{point}  
Over the zero dimensional twisted sectors,
the contribution to $\Phi_n$ are:
\begin{itemize}
\item the contribution from $(C_4', i) \sqcup (C_4', -i) $ is
$$\frac{1}{4} \prod_{j=1,2}\frac{q_j}{(1-q_j)}\cdot
\frac{1- q + q_1+ q_2 - q_1q_2+qq_1+qq_2+qq_1q_2}{(1+q^2)(1+q_1^2)(1+q_2^2)};$$
\item the contribution from $(C_6', \ep^2) \sqcup (C_6', \ep^4)$ is
$$\frac{1}{3} \prod_{j=1,2}\frac{q_j}{(1-q_j)}\cdot
\frac{1-q+( q + 2)(q_1+q_2)+(2q+1)q_1q_2}
{(1+q+q^2)(1+q_1+q_1^2)(1+q_2+q_2^2)};
$$
\item the contribution from $(C_6'', \ep^2) \sqcup (C_6'', \ep^4)$ is
\[
 \begin{split}
  -&\frac{1}{3}\prod_{j=1,2,3} \frac{q_j}{(1-q_j)}\cdot \\
 &\frac{1-q+(q+2)(q_1+q_2+q_3)+(2q+1)(q_1q_2+q_1q_3+q_2q_3)+(q-1)q_1q_2q_3}
  {(1+q+q^2)(1+q_1+q_1^2)(1+q_2+q_2^2)(1+q_3+q_3^2)}.
 \end{split}
\]
\end{itemize}
\end{proposition}

\begin{proof}
To simplify the notation, we will use $(C,\ld)$ to denote a twisted sector.
 
The integrand are determined by the eigenvalues of the bundles involved
in the Riemann Roch formula. 

For $(C_4, \ld), (C_6, \ld)$  of  $I\mbar_{1,1}$,  explicit calculation will show that
the eigenvalues of $L_1$, $H$ and $\Omega_{\mbar_{1,1}}$ are
$\ld$, $\ld$, $\ld^2$  respectively. 
Consider the forgetful map $\pi: \M_{1,k} \to \M_{1,k-1}$, 
as $L_i = \pi^* L_i, i<k$ and $H=\pi^*H$, 
remark \ref{pullback} implies that  the eigenvalues of $L_i$ or $H$ is $\lambda$ 
for $(C, \ld)$.
As $\pi$ is smooth, we have a short exact sequence $ 0 \to \pi^*\Omega_{\mbar_{1,k-1}} \to \Omega_{\mbar_{1,k}} \to \omega_{\pi} \to 0 $ , 
and $\omega_{\pi}$ can be identified with $L_k$. 
It is then easy to see that the eigenvalues of $\Omega_{\mbar_{1,2}}$ are
$\ld$, $\ld^2$ for $(C'_4, \ld)$ or $(C'_6, \ld)$,  and the eigenvalues  
of $\Omega_{\mbar_{1,3}}$ are $\ld, \ld,\ld^2$ for $(C''_6, \ld)$.

From the analysis above, on the twisted sector $(C,\ld)$ of $\mbar_{1,n}$ we have
\begin{align*}
& \tch(\frac{1}{1-q\h^{-1}} \prod_{i = 1}^n ( \frac{1}{1-q_i L_i} - \frac{1}{1-q_i\so} ))
=&\frac{(\ld -1)^{n}}{1-q\ld^{-1}}\prod_{i =1}^n\frac{q_i}{(1-q_i)(1-q_i\ld)},\\
\end{align*}
and 
$$\ttd(\mbar_{1,n}) = \frac{1}{(1-\ld^2)(1-\ld)^{n-1}}=
\frac{1}{(1+\ld)(1-\ld)^{n}}.
$$

The sum of the integral on   
$(C_4', i) \sqcup (C_4, -i) $ is then
$$\frac{1}{4}\sum_{\ld= i, -i}\frac{1}{(1-q\ld^{-1})(1+\ld)}
\prod_{i=1,2}\frac{q_i}{(1-q_i)(1-q_i\ld)},$$
which equals
$$\frac{1}{4} \prod_{j=1,2}\frac{q_j}{(1-q_j)}\cdot
\frac{1- q + q_1+ q_2 - q_1q_2+qq_1+qq_2+qq_1q_2}{(1+q^2)(1+q_1^2)(1+q_2^2)}.$$

The remaining cases also follow directly from our formula of $\tch, \ttd$.
\end{proof}

\subsection{Calculation for $X_1$}

Under the isomorphism $\mbar_{1,1} \simeq \pp(4,6)$, the line bundles $\h$ and $L_1$ all correspond
to $\so(1)$,  
hence $$\chi(\mbar_{1,1}, \h^{ -d}\otimes L_1^{ d_1}) = 
\chi (\pp(4,6), \so(d_1 -d)),$$
and $X_1$ is determined by 
$\chi (\pp(4,6), \so(k)), k \in \mathbb{Z}$.

\begin{lemma} \label{l:mf}
Let $ h^0( \so(k)) =\dim_{\cc}H^0(\pp(4,6), \so(k))$,
then
\[ \sum_{k=0}^{\infty} h^0(\so(k))q^k =\frac{1}{(1-q^4)(1-q^6)}, \]
and $ h^0(\so(k)) = 0$ if $k < 0$.
\end{lemma}
\begin{proof}
As a section of  $\so(k)$ on $\pp(4,6)$ corresponds to a 
polynomial $f(x,y)$ satisfying $f(\ld^4 x, \ld^6 y) = \ld^k f(x,y)$ for 
 any $\ld \in \cc^*$,
monomials $x^ay^b$
such that $4a+6b=k$ form a basis of  $H^0(\pp(4,6), \so(k))$. Therefore,
$h^0( \so(k))$ is given by the 
coefficient of $q^k$ in the power series
$\dfrac{1}{(1-q^4)(1-q^6)}$.
\end{proof}

\begin{proposition} \label{p:gmf}
  $$\chi(\mbar_{1,1}, \frac{1}{1-q\h^{-1}}\frac{1}{1-q_1 L_1})=$$
$$\frac{(1-q q_1)(1- q^4 - q^6 -q_1^2 q^6- q_1^2 q^8 - q_1 ^4 q^8
+ q^2 q_1^2 +q^4 q_1^4 + q^6 q_1^6 + q^8 q_1^8)}
{(1- q^4)(1-q^6)(1-q_1^4)(1-q_1^6)} .$$
\end{proposition}

\begin{proof}

We have $H^1(\pp(4,6), \so(k)) \simeq H^0(\pp(4,6), \so(-10-k))^{^\vee}$ by Serre duality,
so $\chi (\pp(4,6), \so(k))= h^0(\so(k)) -h^0(\so(-k-10))$, and
the proposition now follows from the previous lemma. 
 
\end{proof}

\appendix

\section{A simple proof of Pandharipande's vanishing theorem}
\label{SS:appendix}
The purpose of this appendix is to give a very simple and 
\emph{self-contained} proof of Theorem~\ref{t:rP}, first proved in \cite{rP}.
Recall that the theorem states that at genus zero
\begin{equation}\label{v}
 H^j (\mbar_{0,n}, \otimes_{i=1}^n L_i^{d_i}) =0
\end{equation}
for $j \ge 1$ and $d_i \ge 0$. 
\footnote{The method presented in this appendix can also be used to compute
$H^0(\mbar_{0,n}, \otimes L_i^{ d_i})$.
It is also hoped that this method can help to produce an
$S_n$-equivariant version of our genus zero formula \cite{L0}, which
is needed in the quantum $K$-theory \cite{L} computation of general target 
spaces.}

We will prove \eqref{v} by induction on $n$.
Note that \eqref{v} holds for $n=3$ as $\mbar_{0,3}$ is a point.

For $n > 3$, we first treat the special case that one of the $d_i$ is zero, then reduce the  case that all $d_i > 0$ to that special case.


If one of the $d_i$ is zero,  up to permutation of the marked points 
we can assume $d_n=0$. 
Consider the forgetful map $\pi: \mbar_{0,n} \to \mbar_{0,n-1}$, 
as $R^1 \pi_* (\otimes_{i=1}^{n-1} L_i^{d_i}) =0$ by cohomology and base change
(for $C$ rational and degree of $\so_C(D)$ positive, $H^1(C, \so_C(D))=0$),
we have a degenerated Leray spectral sequence which gives 

\begin{equation*} \label{e:16}
 \begin{split}
  &H^j (\mbar_{0,n}, \otimes_{i=1}^{n-1} L_i^{d_i}) 
 =H^j(\mbar_{0,n-1}, R^0 \pi_*(\otimes_{i=1}^{n-1} L_i^{d_i}) ) \\
 =&H^j \Big ( \mbar_{0,n-1}, (\otimes_{i=1}^{n-1} L_i^{d_i}) 
	\otimes (\so+ \sum_{i, d_i \ne 0} \sum_{m=1}^{d_i} L_i^{-m})  \Big ),
 \end{split}
\end{equation*}
and this is zero by induction. Here 
we used the string equation (Prop. ~\ref{se})
that $K$-theoretically
\[
  \pi_* (\otimes_{i=1}^{n-1}L_i^{d_i}) =(\otimes_{i=1}^{n-1} L_i^{d_i}) 
	\otimes (\so+ \sum_{i, d_i \ne 0} \sum_{m=1}^{d_i} L_i^{-m}) .\]

If all $d_i >0 $,
consider $V := \oplus_{i=1}^n L_i^{d_i}$, note that $\bigwedge^n V$ is  
$\otimes_{i=1}^n {L_i}^{d_i}$.
Choose sections $s_i$ of $L_i$ such that the zero locus of the section 
$\oplus_{i=1}^{n }s_i^{\otimes d_i}$ of 
$V$ is empty.(See the remark below.)
Then  the Koszul complex
\[
 \quad 0 {\to} \mathcal{O} \overset{d}{\to} V 
\overset{d}{\to} \bigwedge^2 V \to \cdots 
\overset{d}{\to} \bigwedge^n V \to 0,
\]
is exact, and
we can compute
$H^*(\mbar_{0,n}, \otimes  {L_i}^{d_i})$ from this resolution.

Form the double complex $(C^p (\underline{U}, \mathcal{K}^q); \delta,
d)_{\{p\ge 0, 0 \le q \le n-1\} }$, where $\underline{U}$ is an affine covering of $\mbar_{0,n}$, $\mathcal{K}^q=\bigwedge ^q V$, 
$C^p$ are the \v{C}ech cochain groups, $\delta$ is the \v{C}ech differential. 

\begin{center}
\setlength{\unitlength}{0.0005in}
%

\begingroup\makeatletter\ifx\SetFigFont\undefined
\def\x#1#2#3#4#5#6#7\relax{\def\x{#1#2#3#4#5#6}}%
\expandafter\x\fmtname xxxxxx\relax \def\y{splain}%
\ifx\x\y   
\gdef\SetFigFont#1#2#3{%
  \ifnum #1<17\tiny\else \ifnum #1<20\small\else
  \ifnum #1<24\normalsize\else \ifnum #1<29\large\else
  \ifnum #1<34\Large\else \ifnum #1<41\LARGE\else
     \huge\fi\fi\fi\fi\fi\fi
  \csname #3\endcsname}%
\else
\gdef\SetFigFont#1#2#3{\begingroup
  \count@#1\relax \ifnum 25<\count@\count@25\fi
  \def\x{\endgroup\@setsize\SetFigFont{#2pt}}%
  \expandafter\x
    \csname \romannumeral\the\count@ pt\expandafter\endcsname
    \csname @\romannumeral\the\count@ pt\endcsname
  \csname #3\endcsname}%
\fi
\fi\endgroup
{\renewcommand{\dashlinestretch}{30}
\begin{picture}(5412,4063)(0,-10)
\path(1050,1315)(1050,1840)
\path(1080.000,1720.000)(1050.000,1840.000)(1020.000,1720.000)
\put(1125,1465){\makebox(0,0)[lb]{\smash{{{\SetFigFont{10}{14.4}{rm}$d$}}}}}
\path(2925,1315)(2925,1840)
\path(2955.000,1720.000)(2925.000,1840.000)(2895.000,1720.000)
\put(3000,1465){\makebox(0,0)[lb]{\smash{{{\SetFigFont{10}{14.4}{rm}$d$}}}}}
\path(2850,2590)(2850,3115)
\path(2880.000,2995.000)(2850.000,3115.000)(2820.000,2995.000)
\put(2925,2740){\makebox(0,0)[lb]{\smash{{{\SetFigFont{12}{14.4}{rm}$d$}}}}}
\path(975,2590)(975,3115)
\path(1005.000,2995.000)(975.000,3115.000)(945.000,2995.000)
\put(1050,2740){\makebox(0,0)[lb]{\smash{{{\SetFigFont{10}{14.4}{rm}$d$}}}}}
\path(3825,865)(4425,865)
\path(4305.000,835.000)(4425.000,865.000)(4305.000,895.000)
\put(3900,1015){\makebox(0,0)[lb]{\smash{{{\SetFigFont{10}{14.4}{rm}$\delta$}}}}}
\path(3825,2140)(4425,2140)
\path(4305.000,2110.000)(4425.000,2140.000)(4305.000,2170.000)
\put(3900,2290){\makebox(0,0)[lb]{\smash{{{\SetFigFont{10}{14.4}{rm}$\delta$}}}}}
\path(1800,2140)(2400,2140)
\path(2280.000,2110.000)(2400.000,2140.000)(2280.000,2170.000)
\put(1875,2290){\makebox(0,0)[lb]{\smash{{{\SetFigFont{10}{14.4}{rm}$\delta$}}}}}
\path(1800,865)(2400,865)
\path(2280.000,835.000)(2400.000,865.000)(2280.000,895.000)
\put(1875,1015){\makebox(0,0)[lb]{\smash{{{\SetFigFont{10}{14.4}{rm}$\delta$}}}}}
\path(25,490)(25,4015)
\path(55.000,3895.000)(25.000,4015.000)(-5.000,3895.000)
\path(25,490)(4400,490)
\path(4280.000,460.000)(4400.000,490.000)(4280.000,520.000)
\put(-350,3940){\makebox(0,0)[lb]{\smash{{{\SetFigFont{10}{14.4}{rm}$q$}}}}}
\put(4525,40){\makebox(0,0)[lb]{\smash{{{\SetFigFont{10}{14.4}{rm}$p$}}}}}
\put(325,790){\makebox(0,0)[lb]{\smash{{{\SetFigFont{8}{14.4}{rm}
$C^0(\underline{U}, \so)$}}}}}
\put(2375,790){\makebox(0,0)[lb]{\smash{{{\SetFigFont{8}{14.4}{rm}
$C^1 (\underline{U}, \so)$}}}}}
\put(325,2140){\makebox(0,0)[lb]{\smash{{{\SetFigFont{8}{14.4}{rm}
$C^0 (\underline{U}, \so(V))$}}}}}
\put(2375,2140){\makebox(0,0)[lb]{\smash{{{\SetFigFont{8}{14.4}{rm}
$C^1(\underline{U}, \so(V))$}}}}}
\end{picture}
}
\end{center}

Using two canonical filtrations (by $p$ and $q$ respectively), we obtain
two spectral sequences ${}'E_r^{p,q}$ and ${}''E_r^{p,q}$ with 
\begin{gather*}
{}'E_1^{p,q} =H^p(\mbar_{0,n}, \mathcal{K}^q),\\
{}''E_2^{p,q} = H^p (\mbar_{0,n}, \mathcal{H}_d^q (\mathcal{K}^*)).
\end{gather*}
These two spectral sequences abut to the same hyper-cohomology
$\mathbb{H}^*(\mbar_{0,n}, \mathcal{K}^*).$

By induction, 
${}'E_1^{p,q} =0$ if $p \ne 0$, since $\mathcal{K}^q$ is the direct sum of
$\otimes L_i^{d_i'}$'s with some $d_i' =0$ . 
So ${}'E_r^{p,q}$ degenerates at $r=1$,  
and by our construction 
$\mathbb{H}^{q}(\mbar_{0,n}, \mathcal{K}^*)
= {}'E_1^{0,q}=0$ when $q \ge n$.

Note that ${}''E_2^{p,q} =0$ if $q \ne n-1$,  
 therefore
${}''E_r^{p,q}$ degenerates at $r=2$, and we have
$H^{p}(\mbar_{0,n} , \otimes L_i ^{d_i})= 
{}''E_2^{p,n-1} 
=\mathbb{H}^{p +n- 1}(\mbar_{0,n}, \mathcal{K}^*)=0$ when $p \ge 1$.

\begin{remark}The zero locus of the section 
$\oplus_{i=1}^{n }s_i^{\otimes d_i}$ is contained in the zero locus of the 
section $\oplus_{i=1}^{n-2} s_i$ of
the vector bundle $\bigoplus_{i=1}^{n-2} L_i$. 
Since having empty zero locus is an open property for sections,
it is easy to show that a generic section of $\bigoplus_{i=1}^{n-2} L_i$
on $\mbar_{0,n}$ has empty zero locus inductively using a forgetful map
as follows.

The statement holds for $n=3, 4$ obviously.
For $n >4$, consider the forgetful map $\pi: \mbar_{0,n} \to \mbar_{0,n-1}$. 
Since $L_i = \pi^* L_i (D_i), 1\le i \le n-2$, 
where $D_i$ is the image of the $i$-th section of $\pi$,
a section $t_i$ of $L_i$ on $\mbar_{0,n-1}$ 
would induce a section $s_i$ of $L_i$
on $\mbar_{0,n}$ with support $\supp s_i= \pi^{-1}(\supp t_i) \cup D_i$.
 It is straightforward to check 
 $\cap_{i=1}^{n-2}\supp s_i = \emptyset$ iff for all $1\le j \le n-2,
 \cap_{i=1, i\ne j}^{n-2} \supp t_i =\emptyset$, 
and these conditions hold for generic $t_i$ by induction.
\end{remark}


\end{document}